\documentclass[a4paper]{article}

\usepackage[english]{babel}
\usepackage[utf8]{inputenc}

\usepackage{amsmath,amssymb,amsthm}

\theoremstyle{definition}
\newtheorem{dfn}{Definition}
\theoremstyle{plain}
\newtheorem{lem}[dfn]{Lemma}
\newtheorem{prop}[dfn]{Proposition}
\newtheorem{col}[dfn]{Corollary}
\theoremstyle{remark}
\newtheorem*{rem}{Remark}

\newcommand\er{\mathbb{R}}
\newcommand\ce{\mathbb{C}}
\newcommand\V{\mathsf{V}}
\newcommand\G{\mathsf{G}}
\newcommand\Tan{\mathcal{T}}
\newcommand\Vecf{\mathcal{X}}
\newcommand\Df{\Omega}
\newcommand\Sf{\Sigma}
\newcommand\Sdf{\Sf\Df}
\newcommand\PSdf{{\Sdf}_0}
\newcommand\Og{\mathsf{O}}
\newcommand\so{\mathfrak{so}}
\newcommand\Cl{\mathsf{Cl}}
\newcommand\Spin{\mathsf{Spin}}
\newcommand\Spnr{\mathsf{S}}
\newcommand\Sb{\mathcal{S}}
\newcommand\Af{\mathsf{A}}
\newcommand\Sa{\Spnr\Af}
\newcommand\PSa{{\Sa}_0}
\newcommand\Alt{\textstyle\bigwedge}
\newcommand\sla{\mathfrak{sl}}
\newcommand\M{\mathcal{M}}
\newcommand\Pb{\mathcal{P}}
\newcommand\eps{\varepsilon}
\newcommand\phpm{\varphi_\pm}
\newcommand\ph{\varphi}
\newcommand\ii{\mathsf{i}}
\newcommand\extp{\wedge}
\newcommand\intp{\mathbin\lrcorner}
\newcommand\clp{\cdot}
\newcommand\clf{{\gamma\clp}}
\newcommand\cle{\clf\extp}
\newcommand\dclf{{\dual\gamma\!\!\clp}}
\newcommand\dcli{\dclf\intp}
\newcommand\conedclf{{\dual{\cone{\gamma}}\!\!\clp}}
\newcommand\conedcli{\conedclf\intp}
\newcommand\dual[1]{{#1}^*}
\newcommand\pb[1]{{#1}^*}
\newcommand\cone[1]{\overline{#1}}
\newcommand\covd{\nabla}
\newcommand\dif{\mathrm{d}}
\newcommand\dr{\dif r}
\newcommand\vr{\partial_r}
\DeclareMathOperator\Twst{T}
\DeclareMathOperator\Dir{D}
\DeclareMathOperator\prj{p}
\DeclareMathOperator\sym{sym}
\DeclareMathOperator\Ker{Ker}
\DeclareMathOperator*\tsum{\textstyle\sum}

\title{Killing spinor-valued forms and the cone construction}

\author{Petr Somberg, Petr Zima}
\date{}
\begin{document}

\maketitle

\begin{abstract}
On a pseudo-Riemannian manifold $\M$ we introduce 
a system of partial differential Killing type equations 
for spinor-valued differential forms, and study
their basic properties. We discuss the relationship 
between solutions of Killing equations on $\M$ and
parallel fields on the metric cone over
$\M$ for spinor-valued forms.
\end{abstract}

\section{Introduction}
\label{sec:intro}

The subject of the present article are the systems of over-determined 
partial differential equations for spinor-valued
differential forms, classified as a type of Killing equations. 
The solution spaces of these systems of PDE's
are termed Killing spinor-valued differential forms. A central 
question in geometry asks for pseudo-Riemannian manifolds admitting 
non-trivial solutions of Killing type equations, namely how the 
properties of Killing spinor-valued forms relate to the underlying 
geometric structure for which they can occur.

Killing spinor-valued forms are closely related to Killing spinors and
Killing forms with Killing vectors as a special example.  Killing
spinors are both twistor spinors and eigenspinors for the Dirac
operator, and real Killing spinors realize the limit case in the
eigenvalue estimates for the Dirac operator on compact Riemannian spin
manifolds of positive scalar curvature. There is a classification of
complete simply connected Riemannian manifolds equipped with real
Killing spinors, leading to the construction of manifolds with the
exceptional holonomy groups $\G_2$ and $\Spin(7)$, see
\cite{friedrich1990}, \cite{bar1993}.  Killing vector fields on a
pseudo-Riemannian manifold are the infinitesimal generators of
isometries, hence they influence its geometrical properties. In
particular, on compact manifolds of negative Ricci curvature there are
no non-trivial Killing vector fields, while on manifolds of
non-positive Ricci curvature are all Killing vector fields parallel. A
generalization of Killing vector fields are Killing forms,
characterized by the fact that their covariant derivative is totally
skew-symmetric tensor field.

There is rather convenient tool allowing to describe invariant 
systems of partial differential equations. It is based on 
Stein-Weiss gradient operators constructed by decomposing 
the covariant derivative into individual invariant components, 
see \cite{stein1968}.  
Among prominent examples of Stein-Weiss gradients are the twistor-like 
operators, which are overdetermined operators corresponding to the 
highest weight gradient component.  It can be shown that the
solution spaces of Killing type equations are always in the kernel of 
corresponding twistor operators.  On the other hand, equations  
given by vanishing of twistor operators are equivalent to weaker systems 
called conformal Killing equations.

The main result of the present article establishes
a correspondence between special solutions of 
Killing equations for spinor-valued forms on the 
base pseudo-Riemannian manifold $\M$ and certain parallel spinor-valued 
forms on the metric cone $\cone{\M}$ over $\M$.  Since
the existence of parallel spinor-valued forms can be 
reduced to a holonomy problem, one can at least partially 
classify the manifolds admitting
these special solutions by enumerating all possible holonomy groups.
In particular, for compact irreducible Riemannian manifolds one can
exploit the Berger-Simmons classification, cf.\ \cite{berger1955} and
\cite{simons1962}.  This was already done by Bär in \cite{bar1993} for
real Killing spinors and by Semmelmann in \cite{semmelmann2003} for
ordinary Killing forms.

Let us briefly describe the content of our article. After general
introduction in Section \ref{sec:prelim}, we employ in Section
\ref{sec:decomp} the representation theory of $\Spin(n_+,n_-)$ and
produce invariant decompositions which offer a deeper insight into
subsequent treatment of spinor-valued forms.  We show that the space
of spinor-valued forms is highly reducible and as a special case we
discuss the primitive spinor-valued forms.  We also introduce the so
called generalized twistor modules as distinguished components in the
decomposition of tensor products with the dual of the fundamental vector
representation.  Then we give rather straightforward definition of
Killing equations on spinor-valued $p$-form fields in Section
\ref{sec:killing} and prove a basic property characterizing its
relationship to other types of Killing equations.  In Section
\ref{sec:cone}, we introduce the metric cone over the base pseudo-Riemannian
manifold $\M$, and discuss the lifts of both spinor-valued form fields
and Killing equations on $\M$ to its metric cone $\cone{\M}$. The
conclusion is that for any degree $p$, there is an injection from
special Killing spinor-valued $p$-forms on $\M$ to parallel
spinor-valued $(p+1)$-forms on $\cone{\M}$ for the Levi-Civita
connection on the metric cone. The variance $\epsilon$ in the signature
$(n_+,n_-)$ for $\M$ and $(\overline{n}_+,\overline{n}_-)$ for
$\cone{\M}$ is built into the definition of Killing number of special
Killing spinor-valued forms on $\M$, and at the same time appears in
the formulas for the connection on the metric cone $\cone{\M}$.

\label{sec:prelim}

\section{Spinor-valued forms}
\label{sec:decomp}
\paragraph{Spinors and form representations.}

The spinor-valued forms originate in the tensor product of forms
(i.e., the alternating tensors) and spinors.  We recall the Clifford
algebra $\Cl(n_+,n_-)$, constructed from the $n$-dimensional
pseudo-Euclidean space $\V = \er^{n_+,n_-} = (\er^n, g)$ equipped with
the standard symmetric bilinear form $g$ of signature $(n_+,n_-)$.
The complex spinor space $\Spnr$ arises as an irreducible complex
$\Cl(n_+,n_-)$-module.  We denote by $\Af^p = \Alt^p\dual\V$ the space
of ordinary forms of degree $p \in \{0,\dots,n\}$, namely $\dual\V =
\Af^1$ denotes the dual of $\V$.  The space of spinor-valued forms of
degree $p$ is defined as the tensor product $\Sa^p = \Af^p \otimes
\Spnr$.

Since $\V$ is naturally embedded into $\Cl(n_+,n_-)$ as its subspace
of generators, the module structure of $\Spnr$ is realized by the
Clifford multiplication `$\clp$' of spinors by vectors.  Equivalently,
the multiplication can be viewed as a $\Cl(n_+,n_-)$-valued $1$-form
denoted $\clf$.  A convenient way to write the defining relations of
the Clifford algebra $\Cl(n_+,n_-)$ is
\begin{align}
\label{eq:clfdef}
  \sym (\clf \otimes \clf) &= -2g,
\end{align}
where $\sym$ denotes the symmetrization over form indices.  
To complete our notation, we
recall the usual exterior product `$\extp$' of two forms, the
interior product `$\intp$' of a vector and a form and finally the
orthogonal dual `$\dual{}$' mapping vectors to $1$-forms
via the isomorphism induced by $g$ .

It is straightforward to verify a few basic relations useful in 
the computations with spinor-valued forms:
\begin{align}
\label{eq:clf}
\begin{aligned}
  X \clp (\cle \Phi) + \cle (X \clp \Phi)
    &= -2\dual X \extp \Phi,
\\
  X \clp (\dcli \Phi) + \dcli (X \clp \Phi)
    &= -2X \intp \Phi,
\\
  X \intp (\cle \Phi) + \cle (X \intp \Phi)
    &= X \clp \Phi,
\\
  \dual{X} \extp (\dcli \Phi) + \dcli (\dual{X} \extp \Phi)
    &= X \clp \Phi,
\end{aligned}
\end{align}
for all $\Phi \in \Sa^p$ and $X \in \V$.  Note that $\clf$ in the
formulas acts simultaneously on the spinor part by the Clifford
multiplication and on the form part by the exterior or interior
product, respectively.


\paragraph{$\Spin$-invariant decompositions.}

Let us briefly recall the case of ordinary forms as discussed in, cf.\
\cite{semmelmann2003}.  The space $\dual\V \otimes \Af^p$ decomposes
with respect to the orthogonal group $\Og(n_+,n_-)$ as
\begin{align}
  \label{eq:derfdec}
  \dual\V \otimes \Af^p
    &\cong \Af^{p-1} \oplus \Af^{p+1} \oplus \Af^{p,1}.
\end{align}
For any $\alpha \in \Af^p$ and $X \in \V$, the projections on the
first two components are given simply by the interior and exterior
products,
\begin{align}
\label{eq:derfprj}
  \prj_1 (\dual{X} \otimes \alpha)
    &= X \intp \alpha,
  &\prj_2 (\dual{X} \otimes \alpha)
    &= \dual{X} \extp \alpha.
\end{align}
Consequently, the remaining component called the 
\emph{twistor module} for $\Af^p$ is the common 
kernel of $\prj_1, \prj_2$
\begin{align}
  \Af^{p,1} &= \Ker (\prj_1) \cap \Ker (\prj_2).
\end{align}

The situation is more complicated for the spinor-valued forms $\Sa^p$.  
Firstly, $\Sa^p$ is reducible with respect to the spin
group $\Spin(n_+,n_-)$.  Its decomposition can be obtained using the
technique of the Howe dual pairs, for details see \cite{slupinski1996}.
The algebraic operators
\begin{align}
  X &= \cle,
  &Y &= -\dcli,
  &H &= [X, Y],
\end{align}
commute with the action of $\Spin(n_+,n_-)$ and span a Lie algebra
isomorphic to the Lie algebra $\sla(2)$.  In particular, we have
\begin{align}
\label{eq:clfcommut}
  H (\Phi) &= (n - 2p)\, \Phi \quad\text{for all}\quad \Phi \in \Sa^p,
\end{align}
and further analysis shows that $\Sa^p$ decomposes as\footnote{
  More precisely, the decomposition is irreducible only for $n$ odd.  
	For $n$ even, each of the summands further decomposes into two irreducible
  components analogously to the decomposition of the spinor space into
  half-spinors $\Spnr = \Spnr^+ \oplus \Spnr^-$.
}
\begin{align}
\label{eq:sfdec}
  \Sa^p &\cong \PSa^0 \oplus \dots \oplus \PSa^l,\quad l = \min \{p,n-p\} .
\end{align}
The component defined as the kernel of $Y$,
\begin{align}
\label{eq:primsf}
  \PSa^q &= \{ \Phi \in \Sa^q \,|\, \dcli \Phi = 0 \}
\end{align}
for $q \in \{0,\dots,[n/2]\}$ is called the space of 
primitive spinor-valued forms.

In order to decompose the space $\dual\V \otimes \Sa^p$, we first
consider projections analogous to \eqref{eq:derfprj} and in addition 
one given by the Clifford multiplication,
\begin{align}
  \prj_1 (\dual{X} \otimes \Phi)
    &= X \intp \Phi,
  &\prj_2 (\dual{X} \otimes \Phi)
    &= \dual{X} \extp \Phi,
  &\prj_3 (\dual{X} \otimes \Phi)
    &= X \clp \Phi,
\end{align}
for $\Phi \in \Sa^p$ and $X \in \V$.  In the case $p=0$ the decomposition degenerates:
\begin{align}
\label{eq:dersdec}
  \dual\V \otimes \Spnr &= \Sa^1 \cong \Spnr \oplus \PSa^1.
\end{align}
The (classical) \emph{twistor module} for $\Spnr$ is just the kernel
of Clifford multiplication,
\begin{align}
  \PSa^1 &= \Ker (\prj_3),
\end{align}
the first two projections being trivial.  The same applies also to the
case $p=n$, $\Sa^n \cong \Sa^0 = \Spnr$.

If $p \in \{1,\dots,n-1\}$, the \emph{twistor module} for $\Sa^p$ is
once again defined as the common kernel of all projections,
\begin{align}
  \Sa^{p,1} &= \Ker (\prj_1) \cap \Ker (\prj_2) \cap \Ker (\prj_3).
\end{align}
However, it turns out that the three projections are not independent,
in fact, the decomposition looks like
\begin{align}
\label{eq:dersfdec}
  \dual\V \otimes \Sa^p
    &\cong (\Sa^{p-1} \oplus \Sa^{p+1} \oplus \Sa^p) \,/\, \Spnr
      \oplus \Sa^{p,1}.
\end{align}
In other words, the first three components share just two copies of
the spinor space $\PSa^0 = \Spnr$.  Hence we need to modify the
projections to make them independent and such a modification is
rather complicated.  Moreover, due to the reducibility of $\Sa^p$ there are
multiplicities in the full decomposition of\/ $\dual\V \otimes \Sa^p$
to irreducible summands and the choice of said modification is not
unique. However, the multiplicities disappear in the restriction 
to the subspace of primitive spinor-valued forms and all 
projections are essentially unique.  For more details and explicit
formulas, see \cite{zima2014}.


\section{Killing equations}
\label{sec:killing}

We assume $(\M,g)$ is an oriented and spin pseudo-Riemannian manifold
of dimension $n$ and signature $(n_+,n_-)$, and $\covd$ is the
Levi-Civita covariant derivative.  As usual we denote the tangent
bundle by $\Tan(\M)$ and the Lie algebra of smooth vector fields by
$\Vecf(\M)$.  We shall consider tensor fields on $\M$ given by smooth
sections of vector bundles associated to a class of
$\Spin(n_+,n_-)$-representations discussed in the previous Section
\ref{sec:decomp}.

\paragraph{Killing forms.}

Killing vector fields can be characterized as the vector fields, 
whose flow preserves the metric $g$.  In terms of the Levi-Civita 
covariant derivative, a Killing vector field $K$ fulfills
\begin{align}
  g (\covd_X K, Y) + g (\covd_Y K, X) &= 0
  \quad \text{for all}\quad X,Y \in \Tan(\M) .
\end{align}
The skew-symmetry of the covariant derivative of $K$ generalizes to the definition of
Killing form as a differential form $\alpha$ fulfilling
(cf.,\ \cite{yano1952}, \cite{semmelmann2003})
\begin{align}
  \label{eq:kf}
  \covd_X \alpha &= \tfrac{1}{p+1}\, X \intp \dif \alpha
\end{align}
for all $X \in \Tan(\M)$, where $p$ is the degree of $\alpha$ and $\dif
\alpha$ denotes the usual exterior differential of $\alpha$.  By the
polarization identity, \eqref{eq:kf} is equivalent to
\begin{align}
  \label{eq:kfpolar}
  X \intp \covd_X \alpha &= 0,
\end{align}
and this implies that the Killing forms provide quadratic first integrals of
the geodesic equation, cf.\ \cite{walker1970}: for $\alpha$ a Killing 
form and $X$ the geodesic vector field on $\M$,
\begin{align}
  \covd_X (X \intp \alpha) &= 0.
\end{align}
By \eqref{eq:derfdec}, the covariant derivative can be decomposed on 
three invariant first-order operators: codifferential
$\dual\dif\colon \Df^p(\M) \to \Df^{p-1}(\M)$
given by $\prj_1 \circ \covd$,
exterior differential
$\dif\colon \Df^p(\M) \to \Df^{p+1}(\M)$
given by $\prj_2 \circ \covd$,
and the twistor operator
$\Twst\colon \Df^p(\M) \to \Df^{p,1}(\M)$
given by projecting $\covd$ on the twistor module.
Here $\Df^p(\M)$ denotes the space of differential forms of degree $p$
and $\Df^{p,1}(\M)$ the space of tensor fields corresponding to the
representation $\Af^{p,1}$. As for Killing forms, \eqref{eq:kf} is 
equivalent to
\begin{align}
  \Twst \alpha &= 0,
  \quad \text{and} \quad
  \dual\dif \alpha = 0.
\end{align}
In particular, Killing forms are in the kernel of the twistor operator, i.e., they
are conformal Killing forms. For more detailed discussion, see \cite{semmelmann2003}.


\paragraph{Killing spinors.}

We denote by $\Pb_\Spin(\M)$ a chosen spin structure on $\M$ and
$\Sb(\M)$ the associated spinor bundle.  The Levi-Civita connection
uniquely lifts to a spin connection and by abuse of notation we denote
the induced covariant derivative on spinors and tensor-spinor fields
$\covd$ as well. $\Sf(\M)$ denotes the space of spinor fields and
$\PSdf^1(\M)$ the space of primitive spinor-valued differential
$1$-forms corresponding to the representation $\PSa^1$.

A Killing spinor is a spinor field
$\Psi$ such that
\begin{align}
\label{eq:ks}
  \covd_X \Psi &= a X \clp \Psi,\quad X \in \Tan(\M),
\end{align}
where $a \in \ce$ is called the Killing number of
$\Psi$. Killing spinors are also intimately related to the 
underlying geometry of $\M$, c.f.\ \cite{friedrich1980}, \cite{baum1991}.

The algebraic decomposition \eqref{eq:dersdec} yields two invariant first-order
differential operators:
Dirac operator
$\Dir\colon \Sf(\M) \to \Sf(\M)$
given by $\prj_3 \circ \covd$, and 
twistor operator
$\Twst\colon \Sf(\M) \to \PSdf^1(\M)$
given by projecting $\covd$ on the twistor module.
The Killing equation
\eqref{eq:ks} is equivalent to
\begin{align}
  \Twst \Psi &= 0,
  \quad \text{and} \quad
  \Dir \Psi = -na \Psi.
\end{align}
In particular, Killing spinors are in the kernel of the twistor operator, 
i.e., they are conformal Killing (or, twistor) spinors.


\paragraph{Killing spinor-valued forms.}

The $\Cl(n_+,n_-)$-valued $1$-form $\clf$, see \eqref{eq:clfdef}, is
invariant for the action of $\Spin(n_+,n_-)$ and hence globally
defined on $\M$. Since the Levi-Civita connection is metric, so is the
spin connection and subsequently we also have $\covd (\clf) = 0$.

\begin{dfn}
A Killing spinor-valued form is a spinor-valued differential form
$\Phi$ of degree $p \in \{1,\dots,n-1\}$ such that
\begin{align}
\label{eq:ksf}
  \covd_X \Phi
    &= a \Big( X \clp \Phi
        - \tfrac{1}{p+1}\, X \intp (\clf \extp \Phi) \Big)
      + \tfrac{1}{p+1}\, X \intp \dif \Phi,\quad X \in \Tan(\M),
\end{align}
where $\dif \Phi$ is the covariant exterior differential 
of\/ $\Phi$ and $a \in \ce$ is called the Killing number of
$\Phi$.
\end{dfn}

The equation \eqref{eq:ksf} first appeared in theoretical physics in
the context of Kaluza-Klein supergravity, cf.\ \cite{duff1983},
\cite{duff1986}.  In geometry, the equation was introduced first in a
simplified form corresponding to $a=0$ in \cite{somberg2011} and in
its general form in \cite{zima2014}.  As in the case of differential
forms, we can reformulate \eqref{eq:ksf} using the polarization
identity:
\begin{align}
\label{eq:ksfpolar}
  X \intp \covd_X \Phi &= aX \intp (X \clp \Phi),
\end{align}
hence Killing spinor-valued forms also yield invariants along
the geodesics of $\M$, but only in the case $a=0$.
A consequence of \eqref{eq:ksfpolar} is
\begin{prop}
Let $\Phi$ be a spinor-valued Killing form with Killing number $a=0$
and let $X$ be the geodesic vector field on~$\M$.  Then
\begin{align}
  \covd_X (X \intp \Phi) &= 0,
\end{align}
i.e., $X \intp \Phi$ is covariantly constant along the geodesics.
\end{prop}
The Killing spinor-valued forms can be directly constructed out of 
Killing forms and Killing spinors.

\begin{prop}
Let $\alpha$ be a Killing form of degree $p \in \{1,\dots,n-1\}$ and
$\Psi$ a Killing spinor with Killing number $a$.  Then $\Phi = \alpha
\otimes \Psi$ is a Killing spinor-valued form with Killing
number $a$.
\end{prop}

\begin{proof}
Let $\{X_1, \dots, X_n\}$ be an orthonormal frame.  We first compute the
exterior covariant derivative of $\Phi$ using \eqref{eq:ks},
\begin{align*}
\begin{split}
  \dif \Phi
    &= \tsum_{i=1}^n \dual{X}_i \extp \covd_{X_i} \Phi
    = \tsum_{i=1}^n \dual{X}_i \extp (
        \covd_{X_i} \alpha \otimes \Psi
        + \alpha \otimes \covd_{X_i} \Psi) =
\\
    &= \dif\alpha \otimes \Psi + a \cle \Phi.
\end{split}
\end{align*}
Now by \eqref{eq:kf} and one more time \eqref{eq:ks}, we get
\begin{align*}
\begin{split}
  \covd_X \Phi
    &= \covd_X \alpha \otimes \Psi + \alpha \otimes \covd_X \Psi
    = \tfrac{1}{p+1}\, X \intp (\dif \alpha \otimes \Psi)
      + a X \clp \Phi =
\\
    &= a \Big(X \clp \Phi
          - \tfrac{1}{p+1}\, X \intp (\cle \Phi) \Big)
      + \tfrac{1}{p+1}\, X \intp \dif \Phi.
\qedhere
\end{split}
\end{align*}
\end{proof}

We conclude this section with a description of the spinor-valued Killing
forms in terms of $\Spin(n_+,n_-)$-invariant first-order operators:
\begin{center}
\renewcommand\arraystretch{1.5}
\begin{tabular}{lll}
  \emph{codifferential}
  & $\dual\dif\colon \Sdf^p(\M) \to \Sdf^{p-1}(\M)$
  & given by $\prj_1 \circ \covd$,
\\[1ex]
  \parbox[m]{8em}{
    \emph{covariant exterior \\ differential}
  }
  & $\dif\colon \Sdf^p(\M) \to \Sdf^{p+1}(\M)$
  & given by $\prj_2 \circ \covd$,
\\[2ex]
  \parbox[m]{8em}{
    \emph{twisted Dirac \\ operator}
   }
  & $\Dir\colon \Sdf^p(\M) \to \Sdf^p(\M)$
  & given by $\prj_3 \circ \covd$,
\\[1ex]
  \emph{Twistor operator}
  & $\Twst\colon \Sdf^p(\M) \to \Sdf^{p,1}(\M)$
  & \parbox[m]{10em}{
    given by projecting $\covd$ \\ on the twistor module .
  }
\end{tabular}
\end{center}
Here $\Sdf^p(\M)$ denotes the space of spinor-valued differential
forms of degree $p$ and $\Sdf^{p,1}(\M)$ the space of tensor-spinor
fields corresponding to the representation $\Sa^{p,1}$.  The equation
\eqref{eq:ksf} is then equivalent to the system of three differential 
equations
\begin{align}
\begin{gathered}
  \Twst \Phi = 0,
  \quad
  \dual\dif \Phi = a \dcli \Phi,
\\
  \text{and} \quad
  \Dir \Phi = \tfrac{1}{p+1} \Big( {-ap(n+2) \Phi}
        -\cle \dual\dif \Phi + \dcli \dif \Phi \Big).
\end{gathered}
\end{align}
In particular, we have
\begin{prop}
Killing spinor-valued forms are in the kernel of the twistor operator,
i.e., they are a special case of conformal Killing spinor-valued
forms.
\end{prop}
For detailed computations and further discussion, see \cite{zima2014}.


\section{The cone construction}
\label{sec:cone}

\paragraph{Metric cone.}

The $\eps$-metric cone over pseudo-Riemannian manifold $(\M,g)$ is the
warped product $(\cone{\M} = \M \times \er_+, \cone{g} = r^2 g +
\eps\, \dr^2)$, where $r$ is the coordinate function on $\er_+$ and
$\eps=\pm 1$.  Note that the signature of $\cone{g}$ is
$(\cone{n}_+,\cone{n}_-)$ with
\begin{align}
  \cone{n}_+ &= n_+ + (1+\eps)/2
  \quad \text{and} \quad \cone{n}_- = n_- + (1-\eps)/2 .  
\end{align}
The canonical projections $\prj_1\colon \cone{\M} \to \M$ and
$\prj_2\colon \cone{\M}\to \er_+$ naturally split the tangent bundle
of $\cone{\M}$ as a direct sum of pull-back bundles
\begin{align}
  \Tan(\cone{\M}) &= \pb\prj_1\Tan(\M) \oplus \pb\prj_2\Tan(\er_+) .
\end{align}
We associate to a vector field $X\in\Vecf(\M)$ or to a $p$-form
$\alpha\in\Df^p(\M)$ a vector field $\cone{X}\in\Vecf(\cone{\M})$ or a
$p$-form $\cone{\alpha}\in\Df^p(\cone{\M})$, respectively, by
\begin{align}
\label{eq:coneform}
  \cone{X} &= \tfrac{1}{r}\, \pb\prj_1 (X) ,
  \quad \cone{\alpha} = r^p\, \pb\prj_1 (\alpha) .
\end{align}
We also denote by $\vr$ and $\dr$ the pull-backs to $\cone{\M}$
of the canonical unit vector field and coordinate 1-form on $\er_+$,
respectively.

In order to express the covariant derivative $\cone{\covd}$ induced by
the Levi-Civita connection on $\cone{\M}$ in terms of $\covd$ on $\M$, we
first compute the commutators
\begin{align}
\label{eq:conecommut}
  [\cone{X}, \cone{Y}] &= \tfrac{1}{r} \cone{[X, Y]},
  \quad [\cone{X}, \vr] = \tfrac{1}{r}\, \cone{X} ,
  \quad \text{for all} \quad X,Y \in \Vecf(\M) .
\end{align}
Subsequently, we have
\begin{align}
\label{eq:conecovdvect}
\begin{aligned}
  & \cone{\covd}_{\cone{X}} \cone{Y}
    = \tfrac{1}{r} (\cone{\covd_X Y} - \eps g (X, Y)\, \vr) ,
  \quad && \cone{\covd}_{\vr} \cone{X} = 0 ,
\\
  & \cone{\covd}_{\cone{X}} \vr
    = \tfrac{1}{r}\, \cone{X} ,
  \quad && \cone{\covd}_{\vr} \vr = 0 ,
\end{aligned}
\end{align}
and dually for $\alpha \in \Df^p(\M)$
\begin{align}
\label{eq:conecovdform}
\begin{aligned}
  & \cone{\covd}_{\cone{X}} \cone{\alpha}
    = \tfrac{1}{r} (\cone{\covd_X \alpha}
        - \dr \extp \cone{(X \intp \alpha)}) ,
  \quad && \cone{\covd}_{\vr} \cone{\alpha} = 0 ,
\\
  & \cone{\covd}_{\cone{X}} (\dr)
    = \tfrac{1}{r}\, \eps \dual{\cone{X}} ,
  \quad && \cone{\covd}_{\vr} (\dr) = 0 .
\end{aligned}
\end{align}

\begin{rem}
It follows from the comparison 
\begin{align*}
\begin{aligned}
  & \widetilde{X} = \pb\prj_1 (X) = r \cone{X} ,
  \quad && \widetilde{\alpha} = \pb\prj_1 (\alpha)
    = \tfrac{1}{r^p}\, \cone{\alpha},
\end{aligned}
\end{align*}
that our formulas \eqref{eq:conecovdvect}, \eqref{eq:conecovdform} are
equivalent to the frequently used formulas 
\begin{align*}
\begin{aligned}
  & \cone{\covd}_{\widetilde{X}} \widetilde{Y}
    = \widetilde{\covd_X Y} - r \eps g (X, Y)\, \vr,
  \quad && \cone{\covd}_{\vr} \widetilde{X}
    = \tfrac{1}{r} \widetilde{X} ,
\\
  & \cone{\covd}_{\widetilde{X}} \widetilde{\alpha}
    = \widetilde{\covd_X \alpha}
      - \tfrac{1}{r} \dr \extp \widetilde{(X \intp \alpha)} ,
  \quad && \cone{\covd}_{\vr} \widetilde{\alpha}
    = -\tfrac{p}{r}\, \widetilde{\alpha} ,
\end{aligned}
\end{align*}
cf.\ \cite{semmelmann2003}.  The advantage of our conventions is that
the lifts of vector fields and differential forms on the cone are always
parallel in the radial direction.  Moreover, the inner product of vector
fields is preserved.
\end{rem}

Let $f = (X_1,\dots,X_n)$ be a local orthonormal frame on $\M$ and
$\omega_i^{jk}$ the corresponding local connection form on
$\M$,\footnote{
  Note that we have raised the index $j$, which corresponds to an
  isomorphism between the Lie algebra $\so(n_+,n_-)$ and the space of
  skew-symmetric bivectors.  This is convenient for subsequent
  computations of the spin connection (cf.\ Lemma
  \ref{lem:conecovdspinor}) without explicit sign changes depending on
  the signature $(n_+,n_-)$.
}
\begin{align}
  \covd_{X_i} Y &= \tsum_{j,k=1}^n
      \omega_i^{jk}\, g(Y,X_j)\, X_k ,
  \quad \text{for all} \quad Y \in \Vecf(\M) .
\end{align}
Then $\cone{f} = (\cone{X}_1, \dots, \cone{X}_n, \vr)$ is a local
orthonormal frame on $\cone{\M}$ and from \eqref{eq:conecovdvect} we get
the corresponding local connection form
$\cone{\omega}_i^{jk}$ on $\cone{\M}$,
\begin{align}
\label{eq:coneconnform}
\begin{aligned}
  & \cone{\omega}_i^{jk}
    = \tfrac{1}{r}\, \omega_i^{jk} ,
  \quad && \cone{\omega}_i^{j(n+1)}
    = - \cone{\omega}_i^{(n+1)j}
    = - \tfrac{1}{r}\, \eps\, \delta_i^j ,
\\
  & \cone{\omega}_{n+1}^{jk} = 0 ,
  \quad && \cone{\omega}_{n+1}^{j(n+1)}
    = - \cone{\omega}_{n+1}^{(n+1)j} = 0 ,
\end{aligned}
\end{align}
using the fact $\cone{g}(\vr,\vr) = \eps$.


\paragraph{Spinors on the cone.}

The Clifford algebra $\Cl(n_+,n_-)$ is a subalgebra of
$\Cl(\cone{n}_+,\cone{n}_-)$ and similarly the spin group
$\Spin(n_+,n_-)$ is a subgroup of\/ $\Spin(\cone{n}_+,\cone{n}_-)$.
The corresponding complex spinor spaces $\Spnr$ and $\cone{\Spnr}$ can
be related as $\Cl(n_+,n_-)$-modules by the isomorphisms
\begin{enumerate}
\item[(a)]
if $n$ is even then
$\cone{\Spnr} \cong \Spnr$,
\item[(b)]
and if $n$ is odd then
$\cone{\Spnr} \cong \Spnr \oplus \widehat{\Spnr}$, where
$\widehat{\Spnr}$ is a second irreducible complex
$\Cl(n_+,n_-)$-module not isomorphic to $\Spnr$.
\end{enumerate}
In both cases there is a unique, up to a $\Cl(n_+,n_-)$-equivariant 
isomorphism,  embedding $\Spnr\subset\cone{\Spnr}$.
We also introduce two other modified embeddings $\phpm \colon \Spnr
\to \cone{\Spnr}$, given by
\begin{align}
\label{eq:phpm}
  \phpm (\Psi) = (1 \mp \sqrt{\eps}\, e_{n+1}) \clp \Psi ,
\end{align}
where $\{e_1,\dots, e_n\}$ is an orthonormal basis of\/ $\V$ and
$\{e_1,\dots, e_n, e_{n+1}\}$ is an orthonormal
basis of\/ $\cone{\V} = \er^{\cone{n}_+,\cone{n}_-}$.
All the subsequent formulas are valid for both choices of the square
root sign, so we choose $\sqrt{\eps} = 1$ for $\eps = 1$ and
$\sqrt{\eps} = \ii$ for $\eps = -1$, respectively.
A straightforward computation based on $(e_{n+1} \clp)^2 = -\eps$ shows
\begin{align}
\begin{aligned}
\label{eq:spindimplus1}
  \phpm (e_i \clp \Psi)
    &= \pm \sqrt{\eps}\, e_i \clp e_{n+1} \clp \phpm (\Psi) ,
\\
  \phpm(e_i \clp e_j \clp \Psi)
    &= e_i \clp e_j \clp \phpm (\Psi),
\\
\ph_+\circ\ph_-(\Psi) &= \ph_-\circ\ph_+(\Psi) =2\Psi .		
\end{aligned}
\end{align}
In particular, the embeddings $\phpm$ are $\Spin(n_+,n_-)$-equivariant and 
injective. In a slightly different notation, this construction can 
be found in \cite{bar1993}, \cite[pp.\ 17--19]{baum1991}.

The cone $\cone{\M}$ is clearly homotopy equivalent to $\M$, hence any
spin structure on $\M$ determines a unique spin structure on
$\cone{\M}$.  In more detail, we construct the spin structure
$\Pb_\Spin(\cone{\M})$ by taking the pull-back of the spin structure
$\Pb_\Spin(\M)$ to $\cone{\M}$ and extending the structure group,
\begin{align}
  \Pb_\Spin(\cone{\M})
    &= \pb\prj_1 \Pb_\Spin(\M)
        \times_{\Spin(n_+,n_-)} \Spin(\cone{n}_+, \cone{n}_-) .
\end{align}
This extension is compatible with the above construction of the
orthonormal frame $\cone{f}$ from $f$, namely, if $f_s$ is a lift of
$f$ then $\cone{f}_s = \pb\prj_1(f_s)$ is a lift of $\cone{f}$.

Hence we can reduce the structure group of natural bundles on the cone
to $\Spin(n_+,n_-)$, in particular, the spinor bundle is given by
\begin{align}
  \Sb(\cone{\M}) &= \pb\prj_1 \Pb_\Spin(\M)
        \times_{\Spin(n_+,n_-)} \cone{\Spnr} ,
\end{align}
and the pull-back $\pb\prj_1 \Sb(\M)$ is canonically a subbundle of
$\Sb(\cone{\M})$.  Now we use the equivariant embeddings $\phpm$
and associate to a spinor field $\Psi \in \Sf(\M)$ spinor fields
$\cone{\Psi}_\pm \in \Sf(\cone{\M})$ by
\begin{align}
  \cone{\Psi}_\pm
    &= (1 \mp \sqrt{\eps}\, \vr) \clp \pb\prj_1 \Psi .
\end{align}
The two choices of the sign yield inequivalent though analogous
results and we shall consider both of them.

\begin{lem}
\label{lem:conecovdspinor}
Let $\Psi$ be a spinor field on $\M$ and\/ $\cone{\Psi}_\pm$ the
associated spinor fields on the cone $\cone{\M}$. The covariant
derivative of\/ $\cone{\Psi}_\pm$ is given by the equations
\begin{align}
\label{eq:conecovdspinor}
  \cone{\covd}_{\cone{X}} (\cone{\Psi}_\pm)
    &= \tfrac{1}{r} \cone{ \Big( \covd_X (\Psi)
        \mp \tfrac{1}{2} \sqrt{\eps}\, X \clp \Psi \Big)}_\pm ,
  \quad \cone{\covd}_{\vr} (\cone{\Psi}_\pm) = 0 ,
\end{align}
for all $X \in \Tan(\M)$.
\end{lem}

\begin{proof}
Let $f$ be a local orthonormal frame field on $\M$ and $f_s$ its lift
to a spin frame field.  The covariant derivative of $\Psi$ is in
general given by
\begin{align}
\label{eq:spinconnform}
  \covd_{X_i} \Psi &= \covd_{X_i} [f_s, s]
    = \Big[ f_s,\; X_i (s)
          + \tfrac{1}{4} \tsum_{j,k}  \omega_i^{jk}\,
            e_j \clp e_k \clp s \Big] ,
\end{align}
where $s$ is the $\Spnr$-valued function which corresponds to
$\Psi$ with respect to $f_s$.  Next let $\cone{f}$ and $\cone{f}_s$ be
the associated frame fields on $\cone{\M}$.  We substitute into
\eqref{eq:spinconnform} the formulas \eqref{eq:coneconnform} for the
connection form on $\cone{\M}$ and compute using
\eqref{eq:spindimplus1}:
\begin{align*}
\newcommand\pbs{\pb\prj_1 s}
\begin{split}
  &\cone{\covd}_{\cone{X}_i} \cone{\Psi}_\pm
    = \cone{\covd}_{\cone{X}_i} \Big[ \cone{f}_s,
        \phpm (\pbs)
      \Big]
    = {}
\\
    &= \Big[ \cone{f}_s,\;
        \cone{X}_i (\phpm (\pbs))
        + \tfrac{1}{4} \tsum_{j,k}
           \cone{\omega}_i^{jk}\,
          e_j \clp e_k \clp \phpm (\pbs)
        + {}
\\
    &\qquad\qquad
        {} + \tfrac{1}{2} \tsum_j
           \cone{\omega}_i^{j(n+1)}\,
          e_j \clp e_{n+1} \clp \phpm (\pbs)
      \Big]
    = {}
\\
    &= \tfrac{1}{r} \Big[ \cone{f}_s,\;
        \pb\prj_1 X_i (\phpm (\pbs))
        + \tfrac{1}{4} \tsum_{j,k}
           \omega_i^{jk}\,
          e_j \clp e_k \clp \phpm (\pbs)
        - {}
\\
    &\qquad\qquad
        {} - \tfrac{1}{2}\, \eps \tsum_j
          \delta_i^j\,
          e_j \clp e_{n+1} \clp \phpm (\pbs)
      \Big]
    = {}
\\
    &= \tfrac{1}{r} \Big[ \cone{f}_s,\;
        \phpm \Big(
          \pb\prj_1 \Big(
            X_i (s)
            + \tfrac{1}{4} \tsum_{j,k}
               \omega_i^{jk}\,
              e_j \clp e_k \clp s
            \mp {}
\\
    &\qquad\qquad
            {} \mp \tfrac{1}{2} \sqrt{\eps}\,
              e_i \clp s
          \Big)
        \Big)
      \Big] = {}
\\
    &=
      \tfrac{1}{r} \cone{\Big(
        \covd_{X_i} (\Psi)
        \mp \tfrac{1}{2} \sqrt{\eps}\, X_i \clp \Psi
      \Big)}_\pm,
\end{split}
\end{align*}
where the indices $i, j, k$ always run through $1,\dots,n$.  The proof
of the second equality is trivial.
\end{proof}

Finally, to a spinor-valued $p$-form $\Phi \in \Sdf^p(\M)$ we
associate spinor-valued $p$-forms $\cone{\Phi}_\pm \in
\Sdf^p(\cone{\M})$ by
\begin{align}
\label{eq:conesf}
  \cone{\Phi}_\pm
    &= r^p (1 \mp \sqrt{\eps}\, \vr) \clp \pb\prj_1 \Phi .
\end{align}
Combining \eqref{eq:conecovdform} and \eqref{eq:conecovdspinor} we get
\begin{align}
\label{eq:conecovdspinform}
\begin{aligned}
  \cone{\covd}_{\cone{X}} \cone{\Phi}_\pm
    &= \tfrac{1}{r} \Big(
        \cone{ (
          \covd_X \Phi
          \mp \tfrac{1}{2} \sqrt{\eps}\, X \clp \Phi )
        }_\pm
        - \dr \extp \cone{(X \intp \Phi)}_\pm
      \Big),
  \quad & \cone{\covd}_{\vr} \cone{\Phi} &= 0,
\end{aligned}
\end{align}
for all $\Phi \in \Sdf^p(\M)$ and $X \in \Tan(\M)$.


\paragraph{Killing equations and the cone.}

The rest of this paper is devoted to the main results which establish
a correspondence between special solutions of the Killing equations on
$\M$ and suitable parallel sections on the cone $\cone{\M}$.

Comparing \eqref{eq:conecovdspinor} with \eqref{eq:ks} we immediately
get the correspondence for Killing spinors, c.f.\ \cite{bar1993} for
the Riemannian case and \cite{bohle2003} for the general
pseudo-Riemannian case.

\begin{col}
Let $\Psi$ be a spinor field on $\M$.  The associated spinor field\/
$\cone{\Psi}_\pm$ on $\cone{\M}$ is parallel if and only if\/ $\Psi$
is a Killing spinor with the Killing number $a = \pm \tfrac{1}{2}
\sqrt{\eps}$.
\end{col}

The next correspondence holds only for special Killing forms
introduced by Tachibana and Yu in \cite{tachibana1970} by an
additional second order condition.  Here we present a slightly
generalized version of Semmelmann's result from \cite{semmelmann2003}
by considering also the case $\eps = -1$.

\begin{dfn}
A special Killing $p$-form is a Killing $p$-form $\alpha$ fulfilling
\begin{align}
\label{eq:skf}
  \covd_X (\dif \alpha) &= b \dual{X} \extp \alpha ,
  \quad \text{for all} \quad X \in \Tan(\M) ,
\end{align}
where $b \in \er$ is arbitrary constant.
\end{dfn}

\begin{prop}
Let $\alpha$ be a $p$-form on $\M$.  The $(p+1)$-form $\beta$ on
$\cone{\M}$ defined by
\begin{align}
  \beta = \dr \extp \cone{\alpha}
      + \tfrac{1}{p+1}\, \cone{\dif \alpha}
\end{align}
is parallel for $\cone{\covd}$ if and only $\alpha$ is a special Killing form with
constant $b = -\eps (p+1)$.
\end{prop}

\begin{proof}
We compute the covariant derivative of $\beta$ using
\eqref{eq:conecovdform}:
\begin{align*}
  & \cone{\covd}_{\cone{X}} \beta
    = \tfrac{1}{r} \Big(
        \dr \extp \Big( \cone{
            \covd_X \alpha
            - \tfrac{1}{p+1}\, X \intp \dif \alpha
          }
        \Big)
        + \cone{
          \tfrac{1}{p+1} \covd_X (\dif \alpha)
          + \eps \dual{X} \extp \alpha
        }
      \Big) ,
\\
  & \cone{\covd}_{\vr} \beta = 0 .
\end{align*}
The claim now follows from \eqref{eq:kf} and \eqref{eq:skf}.
\end{proof}

Analogously to the case of forms itself, the correspondence for
Killing spinor-valued forms holds only for special Killing
spinor-valued forms.  In this case the second order condition which
fits the cone construction has rather complicated form.

\begin{dfn}
A special Killing spinor-valued $p$-form is a Killing spinor-valued
$p$-form $\Phi$ fulfilling
\begin{align}
\label{eq:sksf}
\begin{split}
  & \covd_X (\dif \Phi)
    = b \dual{X} \extp \Phi
      + a \Big(
        X \clp \dif\Phi
        + \tfrac{1}{p+1}\, \cle (X \intp \dif\Phi)
      \Big) + {}
\\ & \quad
      + a^2 \Big(
        2 \dual{X} \extp \Phi
        + \tfrac{2p+1}{p+1}\, \cle (X \clp \Phi)
        + \tfrac{1}{p+1}\, \cle (\cle (X \intp \Phi))
      \Big) ,
\end{split}
\end{align}
for all $X \in \Tan(\M)$, where $a \in \ce$ is the Killing number of
$\Phi$ and $b \in \er$ is another arbitrary constant.
\end{dfn}

The exact form of all the terms containing $\clf$ in both
defining equations \eqref{eq:ksf} and \eqref{eq:sksf} is prescribed
purely by algebraic constraints deduced from the
decomposition \eqref{eq:dersfdec}. For an illustration of the algebraic
constraints in the case of primitive spinor-valued forms, see Lemma \ref{lem:primksf}.

\begin{prop}
\label{prop:coneksf}
Let $\Phi$ be a spinor-valued $p$-form on $\M$. The spinor-valued
$(p+1)$-form\/ $\Xi_\pm$ on the cone $\cone{\M}$ defined by
\begin{align}
\label{eq:coneksf}
  \Xi_\pm &= \dr \extp \cone{\Phi}_\pm
      \mp \tfrac{1}{2(p+1)} \sqrt{\eps}\, \cone{\cle\Phi}_\pm
      + \tfrac{1}{p+1}\, \cone{\dif\Phi}_\pm
\end{align}
is parallel if and only if\/ $\Phi$ is special Killing with Killing
number $a = \pm \tfrac{1}{2} \sqrt{\eps}$ and constant $b = -\eps
(p+1)$.
\end{prop}

\begin{proof}
We compute the covariant derivative of $\Xi_\pm$ using
\eqref{eq:conecovdform} and \eqref{eq:conecovdspinform}:
\begin{align*}
\begin{split}
  & \cone{\covd}_{\cone{X}} \Xi_\pm
    = \tfrac{1}{r} \Big(
        \Big(
          \cone{ (
            \eps\dual{X} \extp \Phi )
          }_\pm
          + \dr \extp \cone{ \Big(
            \covd_X \Phi
            \mp \tfrac{1}{2} \sqrt{\eps}\, X \clp \Phi
           \Big) }_\pm
        \Big) \mp {}
\\ & \quad
        \mp \tfrac{1}{2(p+1)} \sqrt{\eps} \Big(
          \cone{ (
            \covd_X (\cle\Phi)
            \mp \tfrac{1}{2} \sqrt{\eps}\, X \clp (\cle\Phi)
          )}_\pm
          - \dr \extp \cone{(
            X \intp (\cle\Phi)
          )}_\pm
        \Big) + {}
\\ & \quad
        + \tfrac{1}{p+1} \Big(
          \cone{ (
            \covd_X (\dif\Phi)
            \mp \tfrac{1}{2} \sqrt{\eps}\, X \clp \dif\Phi
          ) }_\pm
          - \dr \extp \cone{(
            X \intp \dif\Phi
          )}_\pm
        \Big)
      \Big) ,
\end{split}
\\
  & \cone{\covd}_{\vr} \Xi_\pm = 0 .
\end{align*}
We now separate the radial and tangential components in the first equation
of the last display and get that $\Xi_\pm$ is parallel if and only if for all $X \in
\Tan(\M)$
\begin{align*}
  \covd_X \Phi
    & = \pm \tfrac{1}{2} \sqrt{\eps}\, \Big(
        X \clp \Phi
        - \tfrac{1}{p+1} X \intp (\cle \Phi)
      \Big)
      + \tfrac{1}{p+1}\, X \intp \dif\Phi ,
\\
\begin{split}
  \covd_X (\dif\Phi)
    & = -\eps (p+1)\, \dual{X} \extp \Phi
      \pm \tfrac{1}{2} \sqrt{\eps}\, X \clp \dif\Phi
      \pm \tfrac{1}{2} \sqrt{\eps}\, \cle \covd_X \Phi
      - {}
\\ & \quad
      - \tfrac{1}{4}\, \eps\, X \clp (\cle \Phi) .
\end{split}
\end{align*}
Next we substitute the first equation into the second one and further
rearrange using the relations \eqref{eq:clf}:
\begin{align*}
\begin{split}
  & \covd_X (\dif\Phi)
     = -\eps (p+1)\, \dual{X} \extp \Phi
      \pm \tfrac{1}{2} \sqrt{\eps} \Big(
        X \clp \dif\Phi
        + \tfrac{1}{p+1} \cle (X \intp \dif\Phi)
      \Big) + {}
\\ & \quad
      + \tfrac{1}{4}\, \eps\, \Big(
        \cle \Big(
          X \clp \Phi
          - \tfrac{1}{p+1} X \intp (\cle \Phi)
        \Big)
        - X \clp (\cle \Phi)
      \Big)
\\
     & = -\eps (p+1)\, \dual{X} \extp \Phi
      \pm \tfrac{1}{2} \sqrt{\eps} \Big(
        X \clp \dif\Phi
        + \tfrac{1}{p+1} \cle (X \intp \dif\Phi)
      \Big) + {}
\\ & \quad
      + \tfrac{1}{4}\, \eps \Big(
        2 \dual{X} \extp \Phi
        + \tfrac{2p+1}{p+1}\, \cle (X \clp \Phi)
        + \tfrac{1}{p+1} \cle (\cle (X \intp \Phi))
      \Big) .
\end{split}
\end{align*}
The claim now follows from \eqref{eq:ksf} and \eqref{eq:sksf}.
\end{proof}

Let us recall the notion of primitive Killing spinor-valued $p$-form,
see \eqref{eq:primsf}.  The above correspondence applies also to this
case, and we show that it maps primitive spinor-valued forms back to
primitive spinor-valued forms.

\begin{lem}
\label{lem:primksf}
Let\/ $\Phi$ be a primitive Killing spinor-valued $p$-form on $\M$ with
Killing number $a$.  Then it holds
\begin{align}
\label{eq:primksf}
  \dcli \dif\Phi & = -a (n+2) \Phi .
\end{align}
In particular, we note that\/ $\dif\Phi$ does not need to be primitive.
\end{lem}

\begin{proof}
First recall that $\covd (\clf) = 0$.  Hence the hypothesis implies 
$\covd_X\Phi$ is primitive for all $X \in \Tan(\M)$ and we compute
using \eqref{eq:clf}, \eqref{eq:clfcommut} and \eqref{eq:ksf}:
\begin{align*}
\begin{split}
  0 & = \covd_X (\dcli \Phi) = \dcli \covd_X \Phi
    = {}
\\
    &= a \dcli \Big(
        X \clp \Phi - \tfrac{1}{p+1}\, X \intp (\cle \Phi)
      \Big)
      + \tfrac{1}{p+1}\, \dcli (X \intp \dif\Phi)
    = {}
\\
  & = a \Big(
        {-2 X \intp \Phi} - X \clp (\dcli \Phi)
        + \tfrac{1}{p+1}\, X \intp (\dcli (\cle \Phi))
      \Big) - {}
\\ & \quad
      - \tfrac{1}{p+1}\, X \intp (\dcli \dif\Phi)
    = {}
\\
  & = a \Big(
        {-2 X \intp \Phi}
        - \tfrac{n-2p}{p+1}\, X \intp \Phi
        + \tfrac{1}{p+1}\, X \intp (\cle (\dcli \Phi))
      \Big) - {}
\\ & \quad
      - \tfrac{1}{p+1}\, X \intp (\dcli \dif\Phi)
    = {}
\\
  & = -\tfrac{1}{p+1}\, X \intp (a (n+2) \Phi + \dcli \dif \Phi) ,
\end{split}
\end{align*}
The claim now follows.
\end{proof}

\begin{lem}
Let\/ $\Phi$ be a spinor-valued $p$-form on $\M$ and\/
$\cone{\Phi}_\pm$ the associated spinor-valued $p$-forms on the cone
$\cone{\M}$. Then
\begin{align}
\label{eq:conedcli}
  \conedcli \cone{\Phi}_\pm
    & = \pm \sqrt{\eps}\, \vr \clp \cone{(\dcli \Phi)}_\pm .
\end{align}
\end{lem}

\begin{proof}
We can relate the orthogonal duals of the Clifford multiplication 1-forms
on $\M$ and $\cone{\M}$, respectively, by
\begin{align*}
  \conedcli \pb\prj_1 \Phi
    &= \pb\prj_1 (\dcli \Phi)
      + \eps \vr \clp (\vr \intp \pb\prj_1 \Phi)
    = \pb\prj_1 (\dcli \Phi) .
\end{align*}
Now we substitute \eqref{eq:conesf} and compute using \eqref{eq:clf}
and the fact $(\vr \clp)^2 = -\eps$,
\begin{align*}
\begin{split}
  \conedcli \cone{\Phi}_\pm
    & = r^p\, \conedcli (1 \mp \sqrt{\eps}\, \vr)
          \clp \pb\prj_1 \Phi
    = {}
\\    
    & = r^p ((1 \pm \sqrt{\eps}\, \vr) 
          \clp (\conedcli \pb\prj_1 \Phi)
        \pm 2 \vr \intp \pb\prj_1 \Phi)
    = {}
\\    
    & = r^p (1 \pm \sqrt{\eps}\, \vr) 
          \clp (\conedcli \pb\prj_1 \Phi)
    = {}
\\
    & = \pm \sqrt{\eps}\, r^p\, \vr \clp (1 \mp \sqrt{\eps}\, \vr)
          \clp (\conedcli \pb\prj_1 \Phi)
    = {}
\\
    & = \pm \sqrt{\eps}\, r^p\, \vr \clp (1 \mp \sqrt{\eps}\, \vr)
          \clp \pb\prj_1 (\dcli \Phi)
    = {}
\\
    & = \pm \sqrt{\eps}\, \vr \clp \cone{(\dcli \Phi)}_\pm ,
\end{split}
\end{align*}
proving the claim.
\end{proof}

\begin{prop}
Let\/ $\Phi$ be a primitive Killing spinor-valued $p$-form on $\M$ with
Killing number $a = \pm \tfrac 12 \sqrt{\eps}$.  Then the
spinor-valued differential $(p+1)$-form $\Xi_\pm$ on the cone
$\cone{\M}$ constructed in \eqref{eq:coneksf} is primitive as well.
\end{prop}

\begin{proof}
By \eqref{eq:clf}, \eqref{eq:clfcommut},
\eqref{eq:primksf} and \eqref{eq:conedcli}, we calculate:
\begin{align*}
\begin{split}
  & \conedcli \Xi_\pm
    = \conedcli \Big(
        \dr \extp \cone{\Phi}_\pm
        \mp \tfrac{1}{2(p+1)}\, \sqrt{\eps}\, \cone{\cle\Phi}_\pm
        + \tfrac{1}{p+1}\, \cone{\dif\Phi}_\pm
      \Big)
    = {}
\\
    & = \eps \vr \clp \cone{\Phi}_\pm
      \mp \sqrt{\eps}\, \dr \extp \vr \clp \cone{(\dcli\Phi)}_\pm
      - {}
\\ & \quad
      - \tfrac{1}{2(p+1)}\, \eps \vr \clp \cone{(\dcli (\cle\Phi))}_\pm
      \pm \tfrac{1}{p+1}\, \sqrt{\eps}\, \vr
        \clp \cone{(\dcli \dif\Phi)}_\pm
    = {}
\\
    & = \eps \vr \clp \cone{\Phi}_\pm
      + \tfrac{n-2p}{2(p+1)}\, \eps \vr \clp \cone{\Phi}_\pm
      - {}
\\ & \quad
      - \tfrac{1}{2(p+1)}\, \eps \vr \clp \cone{(\cle (\dcli\Phi)}_\pm
      - \tfrac{n+2}{2(p+1)}\, \eps \vr \clp \cone{\Phi}_\pm
    = 0 .
\end{split}
\end{align*}
The proof is complete.
\end{proof}

\hspace{1cm}

{\bf Acknowledgment:}
The authors gratefully acknowledge the support of the grant GA CR
P201/12/G028 and SVV-2016-260336.

\vspace{0.3cm}

Petr Zima, Petr Somberg

Mathematical Institute of Charles University,

Sokolovská 83, Praha 8 - Karlín, Czech Republic, 

E-mail: zima@karlin.mff.cuni.cz, somberg@karlin.mff.cuni.cz.

\end{document}